\documentclass[preprint,12pt]{elsarticle}

\usepackage[english]{babel}
\usepackage{mathtext,amssymb,amsmath}
\usepackage{amsthm,amsfonts,mathrsfs}
\usepackage{epsfig}
\usepackage{color,graphicx}
\usepackage{psfrag}
\usepackage{float}
\usepackage{lineno}

\newtheorem{theo}{Theorem}[section]
\newtheorem{defi}[theo]{Definition}       
\newtheorem{conj}[theo]{Conjecture}
\newtheorem{prop}[theo]{Proposition}

\newtheorem{lem}[theo]{Lemma}

\newcommand{\Rset}{\mathbb{R}}
\newcommand{\Zset}{\mathbb{Z}}

\newtheorem{thm}{Theorem}
\newcommand{\bthm}{\begin{thm}}
\newcommand{\ethm}{\end{thm}}
\newcommand{\bpr}{\begin{prop}}
\newcommand{\epr}{\end{prop}}
\newtheorem{defn}[theo]{Definition}
\newcommand{\bdf}{\begin{defn}}
\newcommand{\edf}{\end{defn}}
\newcommand{\blm}{\begin{lem}}
\newcommand{\elm}{\end{lem}}
\newtheorem{cor}[theo]{Corollary}
\newcommand{\bcr}{\begin{cor}}
\newcommand{\ecr}{\end{cor}}
\newcommand{\bcj}{\begin{conj}}
\newcommand{\ecj}{\end{conj}}
\newtheorem{rmk}[theo]{Remark}
\newcommand{\brk}{\begin{rmk}}
\newcommand{\erk}{\end{rmk}}
\newtheorem{ex}[theo]{Example}
\newcommand{\bex}{\begin{ex}}
\newcommand{\eex}{\end{ex}}
\newtheorem{genhyp}[theo]{General Hypothesis}

\newcommand{\tb}{Thurston-Bennequin }
\newcommand{\tbrom}{\rm tb}
\newcommand{\T}{T^3}

\newcommand{\rrr}{{\mathbb R}^3 }
\newcommand{\R}{\mathbb{R}}
\newcommand{\ta}{\widetilde\alpha}
\newcommand{\al}{\alpha}
\newcommand{\ha}{\hat\alpha}

\newcommand{\sm}{\smallsetminus}

\newcommand{\Int}{{\rm Int}}

\journal{Topology \& its Applications}

\begin{document}

\begin{frontmatter}




\title{Classical invariants of Legendrian knots in the 3-dimensional torus}




\author{PAUL A. SCHWEITZER, S.J.}

\address{Departamento de Matem\'atica\\ Pontif\'icia Universidade Cat\'olica do Rio de Janeiro, PUC-Rio, Brazil}

\author{F\'ABIO S. SOUZA}

\address{Faculdade de Forma\c{c}\~ao de Professores\\ Universidade do Estado do Rio de Janeiro, UERJ, Brazil}


\begin{abstract}

All knots in $\Rset^3$ possess Seifert surfaces, and so
the classical Thurs\-ton-Bennequin and rotation (or Maslov)
invariants for Legendrian knots in a contact
structure on $\Rset^3$ can be defined.
The definitions extend easily to null-homologous knots in any $3$-manifold
$M$ endowed with a contact structure $\xi$. We generalize the definition of Seifert surfaces and use them to define these invariants for all Legendrian knots, including those that are not null-homologous, in a contact structure on the $3$-torus $T^3$.
We show how to compute the Thurston-Bennequin and rotation
invariants in a tight oriented contact structure on $T^3$ using projections.

\end{abstract}




\end{frontmatter}




\section{Introduction and Statement of Results} \label{introduction}

Let $\xi$ be an {\it oriented contact structure} on a smooth oriented
$3$-manifold $M$,
i.e., a $2$-plane field that locally is the kernel of a totally
non-integrable $1$-form $\omega$, so that locally
$\xi=\ker(\omega)$ with the induced orientation and
$\omega\wedge d\omega$ is non-vanishing. A {\it knot} in a smooth
$3$-manifold $M$ is a smooth embedding $\alpha: S^1\to M$.
A knot in the contact manifold $(M,\xi)$ is {\it Legendrian} if $\alpha$ is everywhere
tangent to $\xi$. Two Legendrian knots $\alpha_0$ and $\alpha_1$
are {\it Legendrian homotopic} if there is a smooth $1$-parameter family
of Legendrian knots $\alpha_t$, $t\in [0,1],$ that connects them.

The Thurston-Bennequin and rotation (or Maslov) numbers are well-known classical invariants of null-homologous oriented Legendrian knots
in $(M^3,\xi)$ that depend only on their Legendrian homotopy class
and, for the rotation invariant, on a fixed Legendrian vector
field \cite{Be, Et}. Given
a Seifert surface $\Sigma$ for the Legendrian knot $\alpha$
in $(M^3,\xi)$, the {\it Thurston-Bennequin invariant} $\tbrom(\alpha)$ is
defined to be the number of times the contact plane $\xi$ rotates relative to the tangent plane to $\Sigma$ in one circuit of $\al$.
The {\it rotation invariant} $r(\alpha)$ is the number of times
the tangent vector $\alpha'$ rotates in $\xi$ relative
to a fixed Legendrian vector field $Z$ in a single circuit of $\al$.
For the standard contact structure $\xi_{std}=\ker(dz-ydx)$ on $\Rset^3$, both invariants of a generic Legendrian knot
can be calculated using the front
and Lagrangian projections of $\Rset^3$ to $\Rset^2$ (see Sections \ref{tbproj} and \ref{maslovproj}).

On the $3$-dimensional
torus $T^3$, we define generalized Seifert surfaces
for knots that are not null-homologous (Definition
\ref{defseifsurf}) and use them to extend the definition of the Thurston-Bennequin invariant to all Legendrian knots in an arbitrary contact structure $\xi$ on $T^3$.
Let $\alpha: S^1=\mathbb{R}/\mathbb{Z}\to T^3$ be a knot in
$T^3$ and let $\widetilde{\alpha}:\mathbb{R}\to\mathbb{R}^3$ be a
lift of $\alpha$ to the universal cover $\widetilde{T^3} =
\mathbb{R}^3$, where $T^3$ is identified with
$\mathbb{R}^3/\mathbb{Z}^3$. If a component of $\widetilde\alpha$ is
compact, then it is a knot on ${\mathbb R^3}$ and
the usual definition of Seifert surfaces applies.
Hence we usually assume the following General Hypothesis,
and then prove the following Proposition.

\begin{genhyp}
\it Each component of $\ta$ is assumed to be non-compact. (Equivalently,
$\alpha$ is not contractible in $T^3$.)
\label{genhyp} \end{genhyp}

Note that the components of $\ta$ are homeomorphic to each other, so if one component is non-compact, then all of them are.
In this case $\ta$ will be periodic, say with smallest period $(p,q,r)\in \Zset^3$.
We define a connected oriented surface $\Sigma \subset
\rrr$ to be a {\it covering Seifert surface} for such a knot $\alpha$ in $T^3$
if $\Sigma$ is $(p,q,r)$-periodic,
$\partial\Sigma$ is one component of $\ta$, and outside a
tubular neighborhood of that component, $\Sigma$ coincides
with an affine half-plane in $\rrr$. Then we call
$\hat\Sigma = \Sigma/\Zset(p,q,r)$ a (generalized) {\it Seifert surface}
for $\al$ (See Definition \ref{defseifsurf}). The arguments using
covering Seifert surfaces $\Sigma\subset \Rset^3$ and Seifert surfaces $\hat\Sigma\subset \Rset^3/\Zset(p,q,r)$
are parallel and equivalent, so some attention is needed to distinguish
the two types of Seifert surfaces.

\begin{prop}
Let $\al$ be a smooth knot on $T^3$ whose lift $\ta$ has all components non-compact and let $\hat\al = \ta/\Zset(p,q,r)$. Then
\begin{enumerate}

\item There is a covering Seifert surface $\Sigma$ with corresponding (generalized) Seifert surface $\hat\Sigma$ for the knot $\al$;

\item If $\Sigma_1$ and $\Sigma_2$ are both covering Seifert surfaces for
the knot $\al$, then the relative rotation number
$\rho(\Sigma_1,\Sigma_2)$, defined to be the number of times
$X_2$ rotates relative to $X_1$ in the normal plane
to $\ta$ in one circuit
of $\al$, where $X_i\ (i=1,2)$ is a unit vector field tangent to $\Sigma_i$
and orthogonal to $\ta'$, is zero.
\end{enumerate}\label{2Seif}
\end{prop}

Using these Seifert surfaces, we extend the definition of the \tb invariant to $(T^3, \xi)$, as follows.
The Thurston-Bennequin invariant $\tbrom(\alpha)$ for a Legendrian knot $\alpha$ in $(T^3, \xi)$ is defined to be the
rotation number of the contact plane $\xi$ with respect to a (generalized)
Seifert surface $\hat\Sigma$ for $\alpha$ in one circuit of $\alpha$ (See
Definition \ref{def:tb}).

\bthm
\begin{enumerate}
\item If $\alpha$ is a Legendrian knot in $(T^3,\xi)$ satisfying the
General Hypothesis \ref{genhyp},
then the Thurston-Bennequin invariant  $\tbrom(\alpha)$ is well-defined.
\item If $\alpha$ satisfies \ref{genhyp} but is null-homologous, then our
definition of $\tbrom(\alpha)$ agrees with the standard definition
of $\tbrom(\alpha)$ in $T^3$.
\item If $\alpha$ does not satisfy \ref{genhyp}, so that $\alpha$
is contractible, then the standard definitions of $\tbrom(\alpha)$ and $\tbrom(\tilde\alpha)$ coincide,
where $\tilde\alpha$ and $\tilde\xi$ are the lifts of $\alpha$
and $\xi$ to the universal cover $(\tilde T^3,\tilde\xi)$
of $(T^3,\xi)$.
\end{enumerate} \label{thm1}
\ethm

\noindent This shows that Definition \ref{def:tb}
extends the usual definition in $\rrr$.
Recall that Kanda \cite{Ka} defines a Thurston-Bennequin invariant
for {\it quasilinear Legendrian knots} on $T^3$, i.e., those that are
isotopic to knots with constant slope, using an
incompressible torus containing the knot to replace the Seifert
surface. In the universal cover, the torus lifts to a plane,
half of which is isotopic to our covering Seifert surface,
so the following result holds.

\bpr For quasilinear knots in $T^3$, our definition of $\tbrom(\alpha)$
agrees with the definition of Kanda \cite{Ka}. \epr

We shall be especially interested in the tight contact structures
$$\xi_n=\ker(\cos(2\pi nz)dx + \sin(2\pi nz)dy)$$
on $T^3=\Rset^3/\Zset^3$, where $n$ is a positive integer. Kanda
\cite{Ka} has shown that for every tight contact structure $\xi$ on $T^3$ there is a contactomorphism  (i.e., a diffeomorphism that preserves the contact structure) from $\xi$ to $\xi_n$, for some $n>0$.

Define projections
$p_{xy}, p_{xz}: T^3\to T^2$ by setting
$p_{xy}(x,y,z)=(x,y)$ and $p_{xz}(x,y,z)=(x,z)$,
where $x,y,z$ are the coordinates modulo $1$ in $T^3$
and in $T^2$.
The projection $p_{xy}$ is called the {\it front projection}, for if we identify
$T^3$ with the space of co-oriented contact elements on $T^2$,
then the wave fronts of the propagation of a wave on $T^2$
are images under $p_{xy}$ of Legendrian curves in $T^3$.
Then a knot in $T^3$ is {\it generic} relative to
both projections $p_{xy}$ and $p_{xz}$
if its curvature vanishes only at isolated points and
the only singularities of the projected knot are transverse
double points and cusps.
Note that every Legendrian
knot can be made generic by an arbitrarily small Legendrian homotopy.
The following Theorem shows how to calculate both invariants of generic
Legendrian knots for $\xi_n$ using the projections
$p_{xy}, p_{xz}: T^3\to T^2$.

\bthm Let $\alpha$ be a generic oriented Legendrian knot in $(T^3,\xi_n)$.
\begin{enumerate}
\item For the projection $p_{xy}$ of $\alpha$,
$\tbrom(\alpha)= P - N + C/2$,
where $P$ and $N$ are the numbers of positive and negative
crossings and $C$ is the number of cusps for
$p_{xy}\circ\alpha$ in one circuit of $\alpha$;
\item For the projection $p_{xz}$ of $\alpha$,
$\tbrom(\alpha)= P - N$,
where $P$ and $N$ are the numbers of positive and negative
crossings for $p_{xz}\circ\alpha$ in one circuit of $\alpha$, and there are no cusps;
\item For the projection $p_{xy}$ of $\alpha$,
the rotation invariant
relative to the Legendrian vector field $Z=\partial/\partial z$ is
$r(\alpha)= 1/2(C_+ - C_-)$,
where $C_+$ and $C_-$ are the numbers of positive and negative
cusps of $p_{xy}\circ\alpha$ in one circuit of $\alpha$;
\item
Let $V=\{t\in S^1\ |\ (x'(t),y'(t))=(0,0) \}$
and suppose that
$2nz(t)\notin {\mathbb Z}$ for every $t\in V$.
Then for the projection $p_{xz}$ of $\alpha$, the rotation invariant
relative to $Z=\partial/\partial z$ is
$r(\alpha)= 1/2\sum_{t\in V} a(t)b(t)$,
where $a(t)=(-1)^{[2nz(t)]}$ and
$b(t)=\pm 1$ according to whether $p_{xz}\circ\alpha'(t)$
is turning in the positive or negative direction in the
$xz$-plane.
\end{enumerate} \label{projectionthm}\ethm

Generalized Seifert surfaces for knots in $T^3$ will be defined
and studied in \S \ref{SeifSurf}. The Thurston-Bennequin invariant
${\tbrom}(\alpha)$ and the rotation invariant $r(\alpha)$
for a generic oriented Legendrian knot in $T^3$ will be treated in \S \ref{tbinv} and \S \ref{s:maslov}, respectively. The proofs of the
four assertions of Theorem \ref{projectionthm} are given in the proofs
of the Propositions \ref{tbpxz}, \ref{tbpxy},
\ref{maslovpxz}, and \ref{maslovpxy}, respectively,
in the Subsections \ref{tbproj} and \ref{maslovproj}.
In the last Section \ref{s:tb-inequality} we calculate the invariants
for quasilinear
Legendrian knots in $(T^3,\xi_n)$ and observe that the Bennequin
inequality for null-homologous Legendrian knots in a tight
contact structure has to be modified in this case. Finally we make a conjecture about the extension of the Bennequin inequality for tight
contact structures on $T^3$.

This paper is a continuation of the work of the second author in his masters thesis \cite{So} at the Pontif\'\i cia Universidade
Cat\'olica of Rio de Janeiro under the direction of the first author.


\section{Seifert surfaces in $T^3$} \label{SeifSurf}

In this section we consider smooth knots in $T^3$ and their Seifert
surfaces, without reference to any contact structure, as a
preparation for studying Legendrian knots and their
Thurston-Bennequin and rotation invariants in the next two sections.


\begin{figure}[H]
\centering     \includegraphics*[width=.81\linewidth]{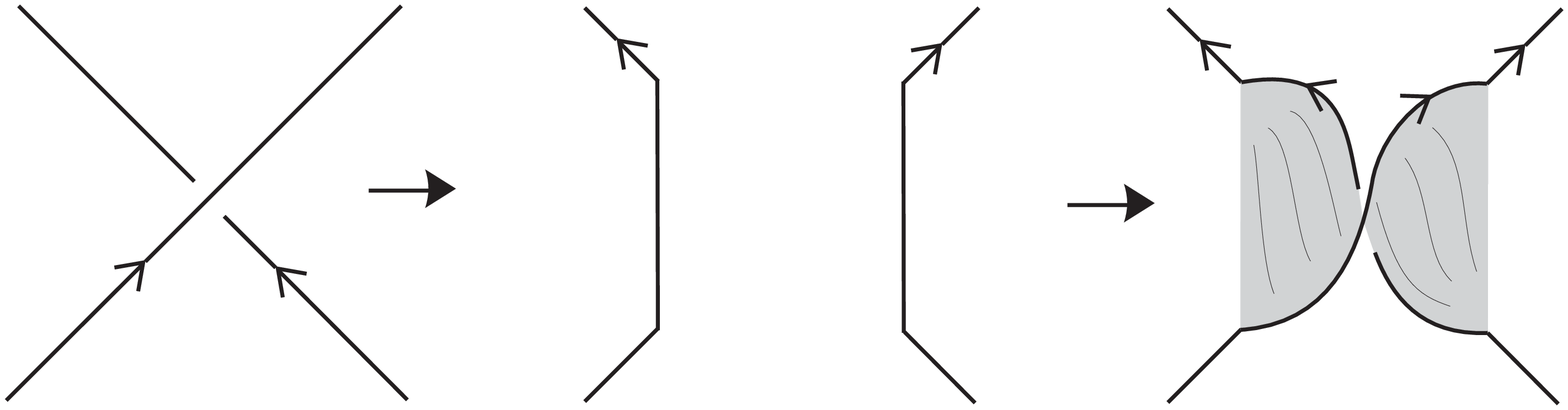}
\caption{Inserting a twisted strip at a crossing.}\label{fig:strip2}
\end{figure}


Recall that a {\it Seifert surface} for an oriented knot (or link)
$\alpha$ in $\mathbb{R}^3$ is a compact connected oriented surface
$\Sigma$ whose boundary is $\alpha$ with the induced orientation.
Every knot and link $\alpha$ has Seifert surfaces, and there is a
well-known method of constructing one using a regular projection
of $\alpha$ in the plane (\cite{BZ}, pp. 16-18). Each crossing is replaced by a
non-crossing that respects the orientation, the resulting circles
are capped off by disjoint embedded disks, and then a twisted
interval is inserted at each crossing, as in Figure
\ref{fig:strip2}. Finally the {\em Seifert circles},
the boundary of the resulting surface, are capped off by
disjoint embedded disks.

Let $\alpha: S^1=\mathbb{R}/\mathbb{Z}\to T^3$ be a knot in
$T^3$ and let $\widetilde{\alpha}:\mathbb{R}\to\mathbb{R}^3$ be a
lift of $\alpha$ to the universal cover $\widetilde{T^3} =
\mathbb{R}^3$, where $T^3$ is identified with
$\mathbb{R}^3/\mathbb{Z}^3$.
Let $(p,q,r)\in {\mathbb Z}^3$ be a generator of
the cyclic group of translations that preserve
$\widetilde{\alpha}$. Note that $(p,q,r)$ is determined up to
multiplication by $\pm 1$ and we choose the sign so that
$\ta(t+1)=\ta(t)+ (p,q,r)$. Any subset of $\rrr$ that is invariant
under this group will be said to be {\it $(p,q,r)$-periodic}. The
following definition adapts the classical concept of Seifert
surfaces for knots in ${\mathbb R}^3$ to the present context.

\begin{defi}
{\rm Let $\al$ be a knot in $T^3$ whose lift $\ta$ has
non-compact components and let $(p,q,r)$ generate the cyclic group of translations that preserve $\ta$.
A smooth surface $\Sigma \subset \mathbb{R}^3$ is a {\it covering Seifert surface} for $\alpha$ if it satisfies the following conditions:}
\begin{enumerate}
{\rm
\item $\Sigma$ is connected, orientable, properly embedded in $\rrr$, and $(p,q,r)$-periodic;

\item $\partial\Sigma = \ta$; and

\item There is an affine half-plane $P_+\subset \rrr$ with boundary a
straight line $S$ such that
$\Sigma$ coincides with $P_+$ outside a
$\delta$-neighborhood $N$ of
$S$, for some sufficiently large $\delta$.
}
\end{enumerate}
{\rm In this case we say that $\hat\Sigma = \Sigma/\Zset(p,q,r)$
is a (generalized) {\it Seifert surface} for $\al$.
If the components of $\ta$ are compact, then a Seifert surface for one of the components can be translated by the action of
${\mathbb Z}^3$ to give a periodic covering Seifert surface.}
\label{defseifsurf} \end{defi}


\begin{figure}[H]
\centering
\includegraphics*[width=.4\linewidth]{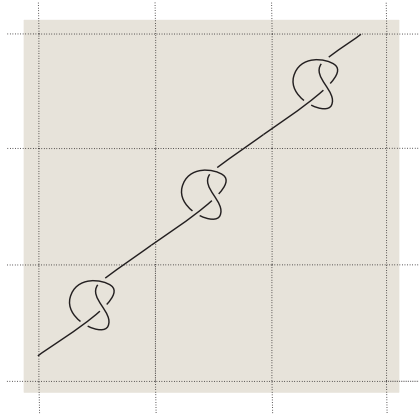}
\caption{A $(p,q,r)$-periodic knot projected into an affine
plane.}
\end{figure}


In this section we shall usually deal with covering Seifert
surfaces, but the same properties could be developed for
Seifert surfaces, and there is a complete correspondence.

Clearly the half-plane $P_+$ and its boundary $S$ are also
$(p,q,r)$-periodic.  It is convenient to choose $P_+$ to be
disjoint from $\ta$ and such that $P_+\subset\Sigma$. Given a covering Seifert surface $\Sigma$ of $\al$, we
define the {\em associated vector field} $X=X(\al,\Sigma)$ along $\ta$
in $\rrr$ to be the $(p,q,r)$-periodic unit vector field along
$\ta$ that is orthogonal to $\ta$, tangent to $\Sigma$, and
directed towards the interior of $\Sigma$.

Given two covering Seifert surfaces $\Sigma_i$ with associated vector
fields $X_i$, $i=1,2$, the {\em relative rotation number} $\rho(\Sigma_1,\Sigma_2)$
of $\Sigma_2$ with respect to $\Sigma_1$ is defined to be the
number of revolutions that $X_2$ makes with respect to $X_1$ in
the positive direction in the normal plane field to $\ta$ along
$\ta$ from a point on $\ta$ to its first $(p,q,r)$-translate in
the positive direction, i.e., in one circuit of $\alpha$.
Note that this number is independent of
the orientation of $\al$, since changing the orientation of $\al$
also changes the orientation of the normal plane.

\begin{proof}[Proof of Proposition \ref{2Seif}.]
First we construct a covering Seifert surface for a
knot $\alpha$ in $T^3$ with period $(p,q,r)$.
Choose a $(p,q,r)$-periodic plane $P$ meeting the proper curve
$\ta$ such that the orthogonal projection of $\ta$ onto $P$ is
regular (i.e., the only singularities are transverse double
points). Fix an orientation of $\ta$. The $(p,q,r)$-periodicity of
both $\ta$ and the affine plane $P$ permits us to adapt the
classical construction of the Seifert surface of a knot in $\rrr$
(see \cite{BZ}, pp. 16-18) in a $(p,q,r)$-periodic fashion. At
each crossing of the image of $\ta$ in $P$,
which we call the the knot diagram,
replace the crossing by two arcs, respecting the orientation of
$\ta$, and insert a twisted strip, as in Figure \ref{fig:strip2}.
Do this so that the resulting collection of ``Seifert curves'' is
pairwise disjoint and $(p,q,r)$-periodic. The Seifert curves that
are simple closed curves are capped off in a periodic fashion by
mutually disjoint disks meeting $P$ only in their boundaries. Then
there will be a number of non-compact proper $(p,q,r)$-periodic
Seifert curves left
over. It is easy to check that this number will be odd, say
$2k+1$, with $k+1$ of them oriented in the positive direction of
$\ta$ and the other $k$ in the opposite direction. (To see this,
consider a plane perpendicular to the direction $(p,q,r)$ that
meets $\ta$ transversely and examine the sign of the intersections
of $\ta$ with this plane.) These curves can be capped off in pairs
with opposite orientations by disjoint oriented periodic infinite
strips, starting with a pair whose projections are adjacent in the
plane $P$. This process will leave one $(p,q,r)$-periodic Seifert curve which can be joined to a half plane $P_+$ contained
in $P$ by another infinite periodic strip so that the result is embedded. The whole construction is done
so as to preserve $(p,q,r)$-periodicity. The result is a covering Seifert surface for $\al$ that coincides with $P_+$ outside a sufficiently large tubular neighborhood $N$ of the line $S$.

Now suppose that $\Sigma_1$ and $\Sigma_2$ are two covering Seifert
surfaces for $\al$. Take a $(p,q,r)$-periodic line $S$ in $\Sigma_1$ and a tubular neighborhood $N$ of $S$ sufficiently large so that the parts of  $\Sigma_1$ and $\Sigma_2$ outside $N$ are half-planes. Remove these half-planes and add an infinite periodic strip in $\partial \bar N$ to connect $\Sigma_1$ and $\Sigma_2$, if necessary. Thus we
obtain a new proper $(p,q,r)$-periodic surface $\Sigma$ which
agrees with the union of $\Sigma_1$ and $\Sigma_2$ inside $N$.
This surface $\Sigma$ will be a proper immersed surface
contained in $\bar N$. Note that
$\Sigma$ projects onto a compact oriented immersed surface on
$T^3$. The following lemma will complete the proof, since
by definition $\rho(\Sigma_1,\Sigma_2)=\rho(X_1,X_2)$.
\end{proof}

\begin{lem} Let $\Sigma$ be a properly immersed $(p,q,r)$-periodic oriented surface
in ${\mathbb R}^3$ whose boundary has two components, one being
$\ta$ with the positive orientation and the other $\ta$ with the
negative orientation, and which projects to a compact surface in
$\widehat T^3={\mathbb
R}^3/{\mathbb Z}(p,q,r)$. Let $X_1$ and $X_2$ be the
vector fields associated to the
two boundary components of $\Sigma$. Then the rotation number
$\rho(X_1,X_2)$ of $X_2$ relative to $X_1$ along $\ta$ is zero.
\label{2boundaries} \end{lem}

\begin{proof}[Proof]
The quotient mappings $\pi': {\mathbb
R}^3\to \widehat T^3$ and $\pi: \widehat T^3\to T^3$ are
projections of covering spaces. Note that $\widehat
T^3$ is diffeomorphic to ${\mathbb R}^2\times S^1$. The curve
$\ta$ projects under $\pi'$ to a compact knot $\ha$ in $\widehat T^3$.
Let $V$ be a small closed tubular neighborhood of $\ha$
(so that $V$ is diffeomorphic
to $S^1\times D^2$, where $D^2$ is the closed unit disk in the
plane) and set $M=\widehat T^3\sm \Int\ V$. Let $\alpha_1$
and  $\alpha_2$ be the loops on the torus $\partial V=\partial M$ obtained by isotoping $\alpha$ in the directions of the vector fields $X_1$ and $X_2$.
We claim that their homology classes satisfy $[\alpha_1] = [\alpha_2]\in H_1(\partial V)$, which implies that the mutual
rotation number $\rho(X_1,X_2)$ vanishes, as claimed.

To see this claim, note that there is a compact oriented surface $\widehat\Sigma$ immersed in $M$ obtained from the projection of
$\Sigma$ into $\widehat T^3$ by a small isotopy
so that its boundary $\partial \widehat\Sigma$ is the union of
$\alpha_1$ and $\alpha_2$ with opposite orientations. Consequently $i''_*([\alpha_1]-[\alpha_2])=0\in H_1M$, where $i'':\partial V\to M$ is the inclusion. Now let $\ell$ and $m$ be the oriented longitude and meridian of $\partial V$, so that there are integers $n_1$ and $n_2$ such that the homology classes of $\alpha_1$ and $\alpha_2$ on $\partial V$ satisfy $[\alpha_r]=[\ell]+n_r[m], r=1,2$.
Since $m$ is contractible on the solid torus $V$, $i'_*([\alpha_1]-[\alpha_2])=0\in H_1V$, where $i':\partial V\to V$ is the inclusion.
In the Mayer-Vietoris exact sequence
$$\cdots \to H_2\widehat T^3 \overset{\partial_2}\to H_1\partial V \xrightarrow[\approx]{i_*} H_1 V\oplus H_1M
\overset{j_*}\to H_1\widehat T^3 \to \cdots$$
 $i_*=i'_*+i''_*$ so $i_*([\alpha_1]-[\alpha_2])=0$.
Since $H_2\widehat T^3\approx H_2({\mathbb
R}^3/{\mathbb Z}(p,q,r))=0$,
$i_*$ is injective, so $[\alpha_1]=[\alpha_2]$, as claimed.
\end{proof}


\section{The Thurston-Bennequin invariant}\label{tbinv}

In this section, we extend the classical definition of the Thurston-Bennequin invariant $\tbrom(\alpha)$ to all
Legendrian knots for an arbitrary contact structure on $T^3$, and
we show how to compute the invariant $\tbrom(\alpha)$
of Legendrian knots in $(T^3,\xi_n)$ using projections.

\subsection{The Thurston-Bennequin invariant for null-homol\-ogous knots.}
First we recall the definition of ${\tbrom}(\alpha)$
for an oriented null-homologous Legendrian knot $\alpha$
relative to a contact structure $\xi$ on an oriented $3$-manifold
$M^3$. Since
$\alpha$ is null-homologous it has a Seifert surface $\Sigma$,
which by definition is an oriented compact connected surface
embedded in $M^3$ with oriented boundary $\alpha$.
Let $X$ and $Y$ be unit
vector fields orthogonal to $\alpha$
(with respect to a metric on $M$), with $X$ tangent to
$\Sigma$ and $Y$ tangent to $\xi$. Then ${\rm tb}(\alpha)$ is defined to be the algebraic number of rotations of $Y$ relative to $X$ in the normal plane field $\alpha^\perp$, which is oriented by the orientations of $M^3$ and $\alpha$, as we make one circuit of $\alpha$ in the positive direction. If we let $\alpha^+$
be a knot obtained by pushing $\alpha$ a short distance
in the direction $Y$, then it is easy to see that
${\rm tb}(\alpha)$ is the intersection number of $\alpha^+$
with $\Sigma$. This is just the linking number of $\alpha^+$
with $\alpha$ because $\Sigma$ is a compact oriented
surface with boundary $\al$.
An argument analogous to the proof of Lemma \ref{2boundaries},
taking $\Sigma$ to be the disjoint union of two Seifert
surfaces $\Sigma_1$ and $\Sigma_2$ for $\alpha$ with opposite
orientations, shows that ${\rm tb}(\alpha)$ is independent of the
choice of the Seifert surface.

\subsection{The Thurston-Bennequin invariant in $T^3$.}
Now consider $T^3$ with
an oriented contact structure $\xi$ and let $\al$ be a Legendrian
knot in $T^3$ with a covering Seifert surface $\Sigma$ for
$\al$ using the lift
$\ta$ to the universal cover ${\mathbb R}^3$. We can
define the rotation number of the lifted contact
structure $\widetilde\xi$ with respect to $\Sigma$ to be the number of rotations of one of the two unit
orthogonal vector fields $Y$ along $\ta$ that is tangent to $\widetilde\xi$ with respect to a unit orthogonal vector field $X$ along $\ta$ that is
tangent to $\Sigma$, in one circuit of $\alpha$.

\begin{defi}
{\rm The {\it \tb invariant} ${\rm tb}(\alpha)$ for a Legendrian knot $\al$ in $T^3$ is the rotation number of the contact structure $\tilde\xi$
with respect to a covering Seifert surface $\Sigma$ for $\al$.}
\label{def:tb}
\end{defi}
It is clear that, on $T^3=\Rset^3/\Zset(p,q,r)$,  ${\rm tb}(\alpha)$
is the rotation number of the induced contact structure $\hat\xi$ with
respect to a (generalized) Seifert surface $\hat\Sigma$ for $\al$
in one circuit of $\hat\al$.
As for knots in $\Rset^3$,
${\rm tb}(\alpha)$ will be the intersection number of $\ta^+$,
the lifted knot $\ta$ pushed a short distance in a direction
transverse to the lifted contact structure $\widetilde\xi$,
with the covering Seifert surface $\Sigma$, in one circuit of $\alpha$.
In this
case, however, if $\alpha$ is not null-homologous, then
${\rm tb}(\alpha)$ is not
a linking number, since $\Sigma$ will not be compact.

We note that Definition \ref{def:tb} is an extension of the
above definition of ${\rm tb}(\alpha)$ for a null-homologous
Legendrian knot $\alpha$ in an oriented $3$-manifold
$M^3$ endowed with a contact
structure $\xi$. As in the null-homologous case, the following holds.

\begin{lem}
The Thurston-Bennequin invariant for a Legendrian knot $\al$
in $\T$ is independent of the choice of the covering Seifert surface
and the orientation of $\alpha$.
\end{lem}

\begin{proof}
According to Proposition \ref{2Seif}, the
rotation number of one covering Seifert surface for the knot $\al$ with
respect to another one is zero. Hence the rotation numbers of the
two covering Seifert surfaces with respect to the contact structure
coincide. Given an orientation of $\alpha$, we choose the orientation of the plane field orthogonal to $\alpha$ such that
the orientations of $\alpha$ and the plane field
determine the standard
orientation of $T^3$, so reversing the orientation of
$\alpha$ reverses the orientation of the plane orthogonal field as well and the rotation number does not change.
\end{proof}

It is worth remarking that our extended definition of ${\rm tb}(\alpha)$
continues to satisfy the usual properties: it does not change if we replace the vector field $Y$ tangent to $\widetilde\xi$ by $-Y$ or by a
vector field $Y^\pitchfork$ transverse to $\widetilde\xi$, or if we use
$-X$ or a vector field $X^\pitchfork$ transverse to $\Sigma$ in
place of $X$.

\subsection{Computation of tb using projections}\label{tbproj}
In this subsection we compute the Thurston-Bennequin invariant ${\rm tb}(\alpha)$ of an oriented Legendrian knot $\alpha$ in $T^3$ relative
to Kanda's tight contact structure $\xi_n$, $n>0$,
using the projections $p_{xy}, p_{xz}: T^3\to T^2$
defined in the Introduction
and a covering Seifert surface $\Sigma$ for
$\alpha$, as defined in Section \ref{SeifSurf}.

First, we recall how to do this for a generic oriented
Legendrian knot $\alpha$ in $\Rset^3$ with the standard contact structure $\xi_{std}$ utilizing its front and Lagrangian projections. The {\it front ({\rm resp.,} Lagrangian) projection} of a Legendrian knot $\alpha$ in $(\Rset^3,\xi_{std})$ is the map $\bar\alpha=pr_F\circ \alpha$ (resp., $\bar\alpha=pr_L\circ \alpha$) where the map $pr_F:\Rset^3\to\Rset^2$ (resp., $pr_L:\Rset^3\to \Rset^2$) is defined by $pr_F(x,y,z)=(x,z)$ (resp., $pr_L(x,y,z)=(x,y)$).


\begin{figure}[H]
\centering
\includegraphics*[width=.49\linewidth]{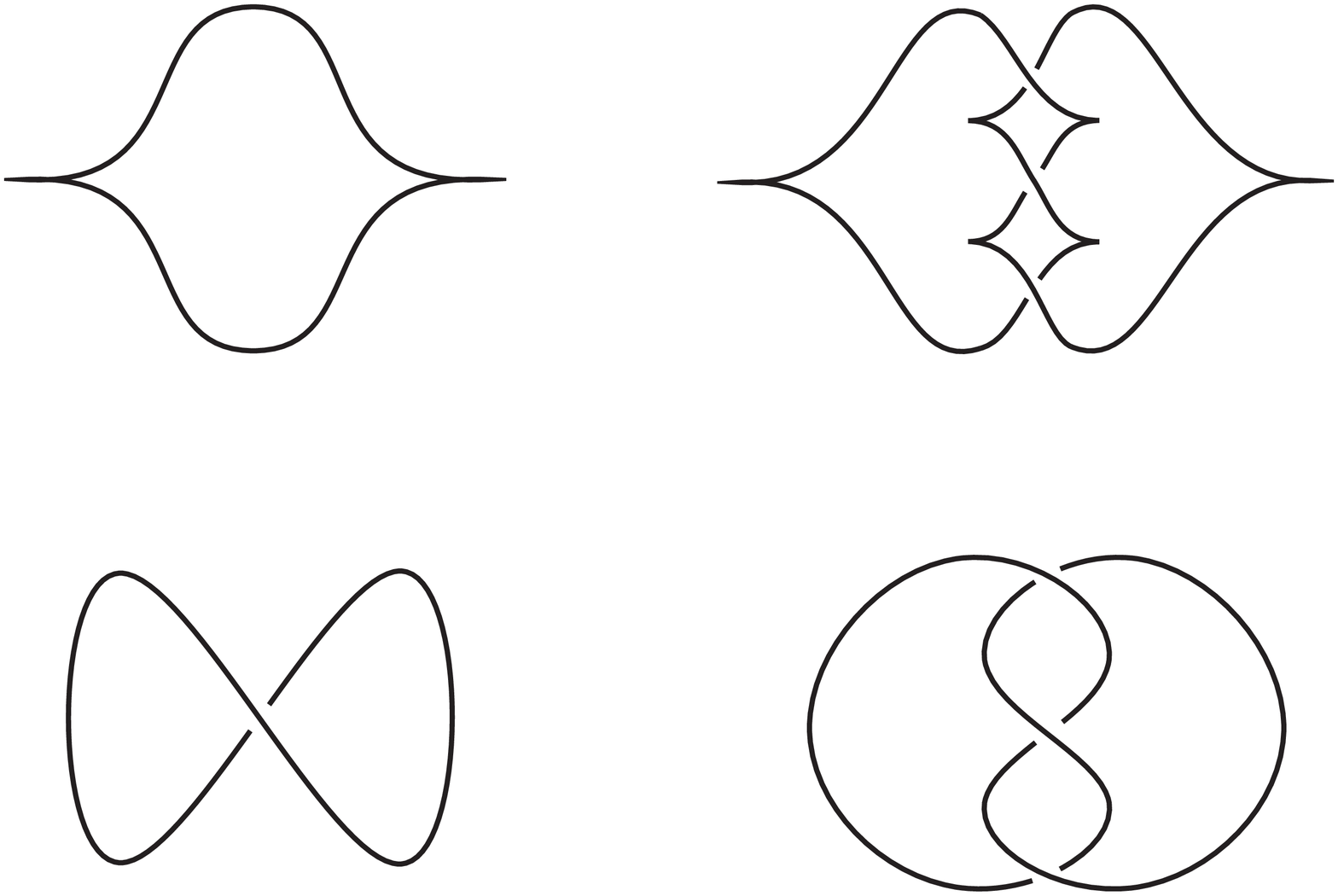}
\caption{Front and Lagrangian projections of a Legendrian unknot and trefoil for $\xi_{std}$ on $\Rset^3$.}\label{fig:projections}
\end{figure}



\begin{figure}[H]
\centering
\psfrag{+}{$+$}
\psfrag{-}{$-$}
\includegraphics*[width=.4\linewidth]{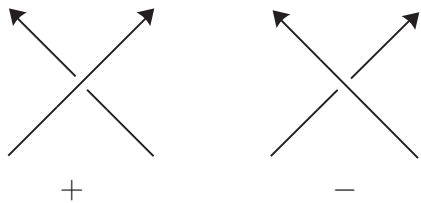}
\caption{Positive and negative crossings.}\label{fig:crossings}
\end{figure}


\bigskip

\noindent{\bf Computation of tb using projections of $\Rset^3$.}
Let
$\bar\alpha=pr_{F}\circ \al$ be the front projection of $\alpha$.
The vector field $Y=\partial/\partial z$ is transverse to $\xi_{std}=\ker(dz-ydx)$ along $\alpha$, and we let
$\alpha^+$ be a knot obtained by shifting $\al$ slightly in the direction $Y$. Then, as observed above, ${\rm tb}(\alpha)$ is
the intersection number of $\al^+$ with the Seifert surface
$\Sigma$, and this is the definition of the
linking number of $\al^+$ with
$\al$. This linking number is known to be half of the algebraic number of crossings of $\bar\alpha$ and $\bar\alpha^+$, where
a crossing is positive if it is right handed and negative if it is left handed (see Figure \ref{fig:crossings}).
One can check this directly by observing the intersections
of $\al^+$ and $\Sigma$, if $\al^+$ is chosen to be slightly
above the $(x,y)$-plane, and the part of $\Sigma$ near to
$\bar\al$ is chosen to be in this plane. Each crossing of
$\bar\al$ will yield one intersection point and contribute
$(+1)$ if the crossing is positive and $(-1)$ if it is negative.


\begin{figure}[H]
\centering
\includegraphics*[width=.7\linewidth]{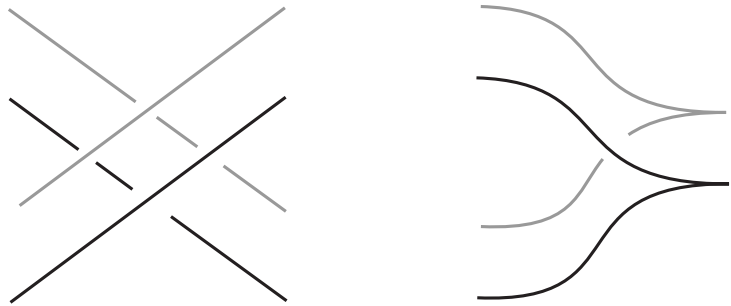}
\caption{The pieces of the curves $\bar\alpha$ and $\bar\alpha^+$.}
\label{fig:shiftingknot}
\end{figure}


The crossings and cusps of $\bar\alpha$ and $\bar\alpha^+$
in the front projection are shown in Figure \ref{fig:shiftingknot}, with $\bar\al$ in black and $\bar\al^+$ in gray.
Each cusp of $\bar\al$ pointing to the left contributes $0$
to the intersection of $\alpha$ and $\Sigma$
since $\alpha^+$ does not meet $\Sigma$ near the cusp,
and a cusp pointing to the right contributes $(-1)$, so two
adjacent cusps contribute $(-1)$. Hence if $P$ and
$N$ are the number of positive and negative crossings
of the front projection $\bar\al$,
respectively, and $C$ is the number of cusps, we conclude that
\begin{equation}\label{eq:two}
{\rm tb}(\alpha)= P - N - C/2.
\end{equation}

In the Lagrangian projection $pr_L(x,y,z)=(x,y)$)
the knots $\alpha$ and $\alpha^+$ project to the same diagram, since $\alpha^+$ is obtained by moving $\alpha$ a small distance in the $Y=\partial/\partial z$-direction. We can see that ${\rm tb}(\alpha)$, the linking number of $\alpha$ and $\alpha^+$, is the algebraic number of the positive and negative crossings of the Lagrangian projection of $\alpha$, ${\rm tb}(\alpha)= P - N$, as in \cite{Et}, p. 13.
\bigskip


\noindent{\bf Computation of tb for Legendrian knots in $(T^3,\xi_n)$.}
Now consider an oriented Legendrian knot $\alpha$ in $(T^3,\xi_n)$ for a fixed $n>0$. Let $$\hat{p}_{xy}, \hat{p}_{xz}:T^2\times\Rset\to T^2$$ be the lifts
to the covering space $T^2\times\Rset$ of the projections
$p_{xy}, p_{xz}:T^3 = T^2\times S^1\to T^2$, where
$\Rset\to S^1$ is the universal cover of the circle.
We shall show how to compute ${\rm tb}(\alpha)$ using the lifted projections $\hat{p}_{xy}, \hat{p}_{xz}$ in a similar way to the case of $({\mathbb R}^3,\xi_{std})$
treated above. We cannot use a linking number here, since the lifted knot $\ha$ does not bound a compact surface, but we can use the intersection number of a perturbed lifted knot $\ha^+$ with the image $\hat\Sigma$
in $T^2\times\Rset$ of the covering Seifert surface of $\Sigma$ of $\alpha$ in ${\mathbb R}^3$ in one circuit of $\alpha$.

\bpr For a generic Legendrian knot $\alpha$ in $(T^3,\xi_n)$
and the projection $p_{xy}$, $$\tbrom(\alpha)=P-N+C/2$$
where $P$ is the number of positive crossings, $N$ is the
number of negative crossings, and $C$ is the number of cusps
of $p_{xy}\circ\alpha$, which must be even,
in one circuit of $\alpha$.  \label{tbpxy}
\epr

As before, since $\alpha$ is generic, the only singularities are
transverse double points and isolated cusps. To determine
which crossings are positive and which are negative
we use a single component of the lift $\ha$ of $\alpha$ to $T^2\times{\mathbb R}$
to see which strand of this component is above and which one is below;
then following $\ha$ from the double point on one arc to the
same double point on the other arc,
the change in the vertical coordinate $z$ determines which arc
is above the other, and hence whether the crossing is positive or
negative (see Figure \ref{fig:crossings}).
Note that each crossing in the projection corresponds to exactly one pair
of strands in $\ha$, and crossings involving two different
components do not contribute anything. The statement of the
Proposition \ref{tbpxy} could just as well be formulated in terms of one period of the lifted curve $\ha$.


\begin{figure}[H]
\centering
\psfrag{1}{$(+1)$}
\psfrag{2}{$(-1)$}
\psfrag{a1}{$\hat\alpha^+$}
\psfrag{S}{$\hat\Sigma$}
      \includegraphics*[width=.7\linewidth]{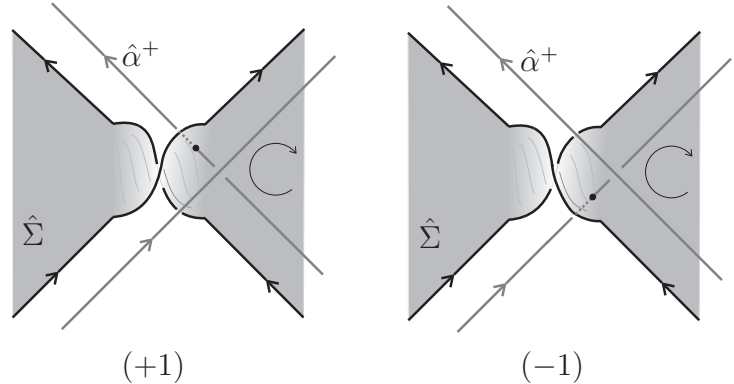}
\caption{The curve $\hat\alpha^+$, in gray, intersects $\hat\Sigma$ positively and negatively close to positive and negative crossings, respectively.}\label{fig:puncturedcrossings}
\end{figure}


\begin{proof}[Proof of Proposition \ref{tbpxy}.]
Lift the contact structure $\xi_n$ to the contact structure
$\hat{\xi}_n$ on $T^2\times{\mathbb R}$. The perpendicular vector field
$\hat Y=(\cos 2\pi nz, \sin 2\pi nz, 0)$ determines the orientation
of $\hat{\xi}_n$. Let $\ha^+$ be a copy of $\ha$ obtained by shifting $\ha$ slightly in the positive direction of $\hat Y$. By Definition \ref{def:tb} the \tb invariant of $\alpha$ is equal to the signed intersection number of $\hat\alpha^+$ with a Seifert surface $\hat\Sigma$ of $\alpha$ in one circuit of $\hat\alpha$.
It is convenient to choose the Seifert surface $\hat\Sigma$
to descend vertically near the Seifert curves, except near the cusp
points on $T^2$, where the covering Seifert surface must move out horizontally
for a small distance before descending. Then it is easy
to check that the contribution of a crossing
will be $(+1)$ for a positive crossing and $(-1)$ for a negative crossing, since the upper strand of
$\hat\alpha^+$ near the crossing will not meet $\hat\Sigma$, and
the lower strand will pierce $\hat\Sigma$ just once, with the
appropriate orientation, as shown in Figure \ref{fig:puncturedcrossings}
for certain typical values of $z$.

The contribution of a cusp point of $\hat p_{xy}\circ \hat\alpha$ is
illustrated in Figure \ref{fig:piercing}, which shows
$\ha$ and $\ha^+$ on $T^2\times\R$.
The arrows show the direction in which
the vertical coordinate $z$ increases. If the vector field $Y$
points to the left of $\ha$ as $\ha$ approaches the cusp point,
as in Figure \ref{fig:piercing} (a), then the
knot $\ha^+$ perturbed in the direction $Y$
will be above $\hat\Sigma$, so there is no intersection
and the contribution will
be $0$. If, on the other hand, $Y$ points toward the right as
$\ha$ approaches the cusp point, then there will be a single
intersection point $p$ where the perturbed knot $\ha^+$ pierces $\hat\Sigma$, and
the contribution will be $+1$, as the orientations in Figure \ref{fig:piercing} (b)
show.


\begin{figure}[H]
\centering
\psfrag{a1}{$\hat\alpha$}
\psfrag{a2}{$\hat\alpha^+$}
\psfrag{p}{$p$}
\psfrag{S}{$\hat\Sigma$}
\psfrag{A}{(a)}
\psfrag{B}{(b)}
      \includegraphics*[width=.62\linewidth]{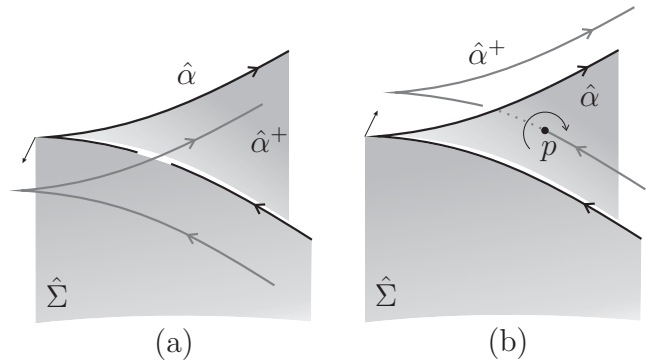}
\caption{In part (a), the perturbed (gray) curve $\hat\alpha^+$  does not intersect $\hat\Sigma$ and, in part (b), $\hat\alpha^+$ intersects $\hat\Sigma$ positively.}\label{fig:piercing}
\end{figure}


The following lemma will complete the proof of the Proposition.
\end{proof}

\blm
The contributions of the cusps alternate between $+1$ and $0$,
so the total contribution of the cusps is $C/2$, where $C$ is
the number of cusps.
\elm

\begin{proof}
If the direction of increasing $z$ is the same from a cusp with value
$+1$ to the next cusp, then $Y$ will point to the left as the next cusp
is approached and its contribution will be $0$.
On the other hand, if the direction of
increasing $z$ reverses, then again the
contribution will be $0$, since the direction of increasing $z$ will be reversed, so $Y$ will point to the right leaving the next cusp
in the direction of increasing $z$. In a similar manner, if a cusp has contribution $0$,
the next cusp will contribute $+1$.
\end{proof}

\bigskip


\noindent{\bf The projection onto the $xz$-plane.}
Now we shall compute $\tbrom(\alpha)$ using the projection $p_{xz}$ of a generic Legendrian knot $\alpha$. Set $\alpha(t)=(x(t),y(t),z(t))$ and note that by a small Legendrian perturbation of
$\alpha$ we can suppose that
\begin{equation}
\hskip 2cm 2nz(t)\notin \Zset {\rm\ \ whenever}\ \ (x'(t),y'(t))= (0,0).\hskip 3cm
\label{znonzero} \end{equation}

In other words, for these values of
$t$ the vertical component of $\alpha'(t)$ is not zero. For other
values of $t$, the plane $\xi_n$ of the contact structure projects
onto the tangent plane of $T^2$ under $p_{xz}$, and so the image
$p_{xz}\circ\alpha$ is a smooth non-singular curve.

The argument used for the projection $p_{xy}$ shows the following result. By analogy with the previous analysis, we let the covering Seifert
surface move off the knot $\ha$ in the direction of the $y$-axis,
instead of the $z$-axis. To determine whether a crossing is positive
or negative, lift the knot to $S^1\times {\mathbb R}\times S^1$,
where the order of the arcs passing through a double point is
well defined.

\bpr For a generic Legendrian knot $\alpha$ in $(T^3,\xi_n)$
that satisfies (\ref{znonzero}), $$\tbrom(\alpha)=P-N$$
where $P$ is the number of positive crossings and $N$ is the
number of negative crossings
of $p_{xz}\circ\alpha$ in one circuit of $\alpha$.  \label{tbpxz}
\epr
It is possible to calculate $\tbrom(\alpha)$ for a
generic Legendrian knot $\alpha$ that does not
satisfy (\ref{znonzero}) using its projection in the
$xz$-plane, but the formula is more complicated, so
we omit it.

These calculations prove the first two assertions of Theorem
\ref{projectionthm}, which relate to the \tb invariant.


\section{The Rotation Number}\label{s:maslov}

Recall that a null-homologous oriented Legendrian knot $\alpha$ in
a $3$-manifold $M$ with an oriented contact structure $\xi$ has an
invariant $r(\alpha)$, called the rotation (or Maslov) number, which
depends on the choice of a non-vanishing section $Z$ of $\xi$.
If $\al$ is null-homologous, then it has a covering Seifert surface $\Sigma$,
and the vector field $Z$ can be determined (up to Legendrian homotopy)
by requiring that it extend to a non-vanishing section of $\xi$
over $\Sigma$.

\bdf
{\rm The {\it rotation number} (or Maslov number) of the oriented Legendrian knot $\alpha$,
$r(\alpha)$, is the
algebraic number of rotations of the tangent vector $\alpha'$ with
respect to $Z$ in the plane field
$\xi$ in a single circuit of $\alpha$.}
\edf

\bpr If $\alpha$ is a null homologous oriented Legendrian knot,
the rotation number $r(\alpha)$ does not depend on the section $Z$.
Furthermore, two Legendrian knots that are isotopic through Legendrian knots have the same rotation number with respect to the same section $Z$. \label{independ}
\epr

\begin{proof} The second affirmation in obvious, since the
rotation number is an integer that varies continuously as
the Legendrian knot varies.

Now let $Z'$ be another global section of $\xi$ and let
$f:M\to S^1$ be the function which gives the angle from $Z$ to
$Z'$. Since $\alpha$ is null homologous in $M$, the image
$f_*[\alpha]$ of its homology class must vanish in $H_1(S^1)$, so
the mutual rotation number of $Z'$ relative to $Z$ is $0$.
Hence the rotation numbers are the same.
\end{proof}

Consequently $r(\alpha)$ for a null-homologous knot $\alpha$
depends only on the orientations of $\alpha$ and $\xi$. Reversing one of these orientations
changes the sign of $r(\alpha)$.
As in the case of the definition of the
Thurston-Bennequin invariant, the rotation number can also be
defined for non-null homologous oriented Legendrian knots, but then it does depend on the choice of the section $Z$ of $\xi$.

This dependence holds, in particular, when
$M=T^3$ (see \cite{Gh}), but for Kanda's tight contact structure
$\xi_n$, we can use the covering Seifert
surface of $\al$ to determine $Z|_{\ta}$ up to Legendrian homotopy.

\blm Let $\al$ be an oriented Legendrian knot
for the contact structure $\xi_n$ on $T^3$,
and let $\Sigma\subset \Rset^3 = \tilde T^3$ be a covering Seifert
surface for $\al$ containing an affine half-plane $P\subset\Rset^3$.
Then there is a non-vanishing Legendrian vector field $Z$ in
$\xi_n|\Sigma$ whose restriction $Z|_P$ is a section of $\xi_n \cap P$.
Furthermore along $\ta$, the restriction $Z|_{\ta}$ is
unique up to periodic Legendrian homotopy. \label{chooseZ} \elm

\begin{proof} If the vertical Legendrian vector field $\partial/\partial z$ is in $P$, then along $P$, $Z=\partial/\partial z$ is a
section of $\xi_n\cap P$.
Furthermore, $\xi_n$ and $P$ are transverse except along isolated
values of $z$, so by continuity the section $Z$ is determined up to
multiplication by a non-vanishing function.
If $\partial/\partial z$ is not in $P$, then $\xi_n$ and $P$ are
transverse, so again $Z$ is determined as a section of
the line field $\xi_n\cap P$ on $P$.
Now extend $Z$ arbitrarily to a periodic and non-vanishing
vector field tangent to $\Sigma$. Clearly $Z$ is
determined up to periodic Legendrian homotopy on $P$, and the usual
argument shows that the restriction $Z|_{\ta}$ is also determined
up to periodic Legendrian homotopy along $\ta$.
\end{proof}

It is interesting to note that on $P$ the restriction $Z|_P$ is
periodically Legendrian homotopic to the vertical Legendrian vector field $\partial/\partial z$, since, in the second case of the above proof,
the angle between them is never $\pi$.

\bdf {\rm The {\it rotation number} (or Maslov number) of an oriented Legendrian knot $\alpha$,
$r(\alpha)$, on the contact manifold $(T^3,\xi_n)$ is the
algebraic number of rotations in the plane field
$\xi_n$ of the tangent vector $\alpha'$ with
respect to the vector field $Z$ given by Lemma \ref{chooseZ}, in a single circuit of $\alpha$.} \label{r on xin}
\edf


\subsection{Computation of the rotation invariant using projections}\label{maslovproj}

\noindent
{\bf Computation of $r$ for Legendrian knots in $(\Rset^3,\xi_{std})$.} Let $\alpha$ be an oriented generic Legendrian knot in the standard contact structure $\xi_{std}=\ker(dz-ydx)$ on $\Rset^3$. In order to calculate the rotation number $r(\alpha)$ we fix the Legendrian vector field $Y=\partial/\partial y$. Then $r(\alpha)$ is the algebraic number of times the field of tangent vectors $\alpha'$ rotates in $\xi_{std}$ relative to $Y$, so $r(\alpha)$ can be obtained by counting how
many times $\alpha'$ and $\pm Y$ point in the same direction. The sign is determined by whether $\alpha'$ passes $\pm Y$ counterclockwise (+1) or clockwise (-1), and then we must divide by two, since in one rotation
$\alpha'$ passes both $Y$ and $-Y$.


\begin{figure}[H]
\centering
      \includegraphics*[width=.85\linewidth]{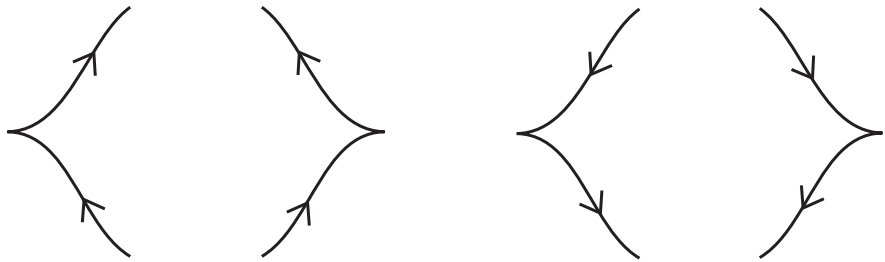}
\caption{Up cusps and down cusps.}\label{fig:cusps}
\end{figure}


If $\bar\alpha$ denotes the front projection of $\alpha$, the field of tangent vectors to $\bar\alpha$, $\bar\alpha'$, points in the direction of $\pm Y=\pm \partial/\partial y$ at the cusps,
which are horizontal in the $xz$-plane. Let us analyze the
upwards left-pointing cusp, the first of the four cusps in
Figure \ref{fig:cusps}. The value of $y$ is just the slope of
$\bar\alpha'$, so $y$ is negative before the cusp and becomes positive,
and thus at the cusp $y(t)$ is increasing so $y'(t)$
is positive and $\alpha'$ passes $+Y$ at the cusp point.
Before the cusp, $x$ is decreasing,
so $x'(t)$ passes from negative to positive at the cusp.
Thus the vector $\bar\alpha'(t)$ turns in the negative direction,
and the contribution is $(-1)$.
By a similar analysis of the other three cases, we see that
a cusp going upwards (the first two cusps in the figure) contributes $(-1)$, while a cusp going downwards (the third and fourth cusps in
the figure) contributes $(+1)$.
Therefore we have shown that the rotation number of $\alpha$ in the front projection is
$$r(\alpha)=1/2(C_d-C_u),$$
where $C_u$ is the number of up cusps and $C_d$ is the number of down cusps in the front projection of $\alpha$. Since $\alpha$ is null-homologous, $r(\alpha)$ does not depend on the choice of
the vector field $Y$, as we observed above.

In the Lagrangian projection $pr_L(x,y,z)=(x,y)$, the vector field $Y$ projects to $\partial/\partial y$, thus the rotation number of $\alpha$ is simply the winding number of the field of tangent vectors of the Lagrangian projection $pr_L\circ \alpha$
of $\alpha$ in $\xi_{std}$,
$$r(\alpha)={\rm winding} (pr_L(\alpha)).$$


\medskip
\noindent
{\bf Computation of $r$ for Legendrian knots in $(T^3,\xi_n)$.}
Let $\alpha$ be an oriented generic Legendrian knot in $T^3$
for the tight contact structure $\xi_n$.
We shall calculate the rotation invariant
$r(\al)$ relative to the vertical vector field
$Z=\partial/\partial z\in \xi_n$.


\begin{figure}[H]
\centering
\psfrag{1}{$(-1)$}
\psfrag{2}{$(+1)$}
      \includegraphics*[width=.45\linewidth]{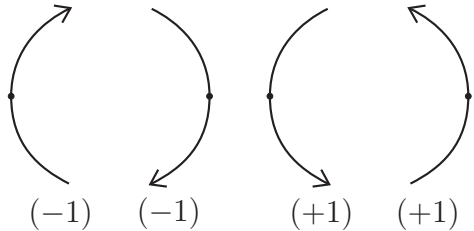}
\caption{Values of $b(t)$ in the projection $p_{xz}$.}\label{fig:last}
\end{figure}



\noindent{\bf The projection $p_{xz}$.}
First we use the projection $p_{xz}: T^3\to T^2$.
By a small Legendrian perturbation, if necessary, we guarantee that
if $2nz(t)\in {\mathbb Z}$ then
the tangent vector $\alpha'(t)$ is not vertical,
i.e., $(x'(t),y'(t))\neq (0,0)$. The only contributions to the rotation number $r(\al)$ occur for points where $\alpha'(t)$ is vertical, and then since $2nz(t)
\notin{\mathbb Z}$ the projection $p_{xz}$
takes $\xi_n$ onto the tangent plane to $T^2$.
Near to where $\alpha'(t)$ is vertical the tangent vector,
which must be non-zero, will be
turning in either the positive direction with respect to the orientation
of the $xz$-plane
and pass the vertical line in the positive direction, and then
we set $b(t)= +1$ (as in the last two cases of Figure
\ref{fig:last}), or in the negative direction (as in the first two cases) where we set $b(t)=-1$.
Let $a(t)=(-1)^{[2nz(t)]}$, where the brackets indicate the
largest integer function, so that $a(t)$ is positive where
the projection $p_{xz}$ of $\xi_n$ onto the tangent $xz$-plane
preserves the orientation and negative where the orientation is reversed, except when $2nz(t)\in {\mathbb Z}$, but we have guaranteed that then
$\alpha'$ will not be vertical.
Thus the contribution of a point where $\alpha'$ is vertical
is half the product of $a(t)$ and $b(t)$.
We have shown the following.

\bpr The rotation invariant $r(\al)$ of a generic oriented knot
$\alpha$ in $T^3$ with respect to the projection
$p_{xz}: T^3\to T^2$ is
$$r(\al)=1/2\sum_{t\in V} a(t)b(t)$$
where $V=\{t\in S^1\ |\ (x'(t),y'(t))=(0,0) \}$
in one circuit of $\alpha$,
provided that $(x'(t),y'(t))\neq (0,0)$ whenever
$2nz(t)\in {\mathbb Z}$. \label{maslovpxy}
\epr


\noindent{\bf The projection $p_{xy}$.} For the projection $p_{xy}:T^3\to T^2$, we must count how many times the tangent field of $\bar{\alpha}=p_{xy}\circ\alpha$
and $Z=\partial/\partial z$ point in the same direction,
and this will happen where $\bar\alpha$ has a cusp
since the projection $\bar\alpha$ will have velocity
$\bar\alpha'(t_0)=0$ at such a point.
Observe that the horizontal normal vector
$Y=(\cos 2\pi nz,\sin 2\pi nz,0)$, which determines the
orientation of the perpendicular contact plane $\xi_n$, projects to
a vector $\bar Y=p_{xy*}(Y)=(\cos 2\pi nz,\sin 2\pi nz)$
perpendicular to the
line tangent to the cusp in the $xy$-plane. The slope of this
line, determined by the value
of $z(t_0)$ at the cusp, may have any value.


\begin{figure}[H]
\centering
\psfrag{a}{$\bar\alpha$}
\psfrag{Y}{$Y$}
      \includegraphics*[width=.26\linewidth]{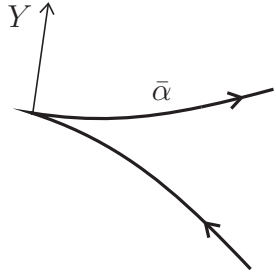}
\caption{A cusp of $p_{xz}\circ\alpha$.}\label{fig:cusp_orientation}
\end{figure}


Consider the orientation of $\alpha$ and the direction of $Y$ in
Figure \ref{fig:cusp_orientation}.
Since the tangent vector $\bar\alpha'(t)$ is turning in the positive
direction in the $xy$-plane, $z'(t_0)>0$ at the cusp.
Before the cusp $\bar\alpha'(t)$ is directed toward the cusp,
and afterwards, it is directed away from the cusp.
Hence it is clear that
$\alpha'(t)$ passes the vertical vector $Z$ in the positive
direction in the contact plane $\xi_n$, so in this case
the contribution of the cusp is $+1$, and we call the cusp {\em positive}. In this case
the projection of $\bar\alpha'(t)$ onto the line through $Y$
has the same direction at $Y$, both before and after the cusp point.
The result is the same
if the diagram in Figure \ref{fig:cusp_orientation}
is rotated in the $xy$-plane.


\begin{figure}[H]
\centering
\psfrag{a}{$\bar\alpha$}
\psfrag{Y}{$Y$}
\psfrag{-}{$-$}
\psfrag{+}{$+$}
      \includegraphics*[width=.8\linewidth]{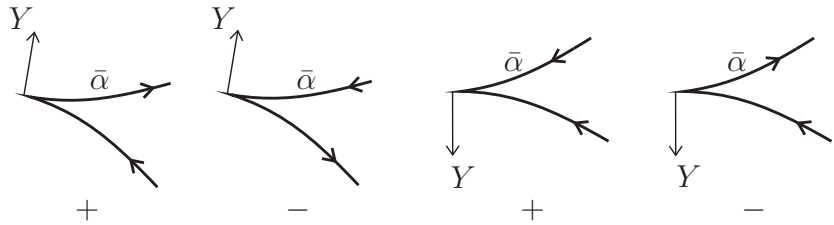}
\caption{Positive and negative cusps for the projection $p_{xz}$.}\label{fig:typesofcusps}
\end{figure}


Now it is clear that if the orientation of $\alpha$ or the direction
of $Y$ is reversed, the sign of the contribution of the cusp
changes. It follows that in all four cases of the orientation
of $\alpha$ and the perpendicular direction of $Y$, the contribution of
the cusp is $+1$ and the cusp is {\em positive} if the projection
of $\bar\alpha'(t)$ onto the line through $Y$ both before
and after the cusp has the same direction as $Y$, and the
cusp is {\em negative}, with contribution $-1$, if the direction
is opposite to $Y$, as shown in Figure \ref{fig:typesofcusps}.

\noindent Thus we have shown the following.

\bpr The rotation number of a generic oriented knot $\alpha$ in
$(T^3,\xi_n)$ with respect to the projection $p_{xy}$ is
$$r(\alpha)=1/2(C_+-C_-)$$
where $C_+$ is the number of positive cusps and $C_-$ is the number of negative cusps of $p_{xy}\circ\alpha$ in one circuit of $\alpha$.
\label{maslovpxz} \epr
This completes the calculation of the rotation invariant
using the projections $p_{xz}$ and $p_{xy}$ as in Theorem
\ref{projectionthm}, so its proof is complete.


\section{Does a Bennequin inequality hold?}\label{s:tb-inequality}
For a null-homologous Legendrian knot $\alpha$ on a tight contact $3$-manifold
$(M,\xi)$ with Seifert surface $\Sigma$, the
Thurston-Bennequin inequality
\begin{equation} tb(\al) + |r(\al)| \leq -\chi(\Sigma) \label{Ben_classical}
\end{equation}
gives an upper bound on $tb(\al)$, provided that
$\chi(\Sigma)\leq 0$ \cite{E,Et}.
It is natural to ask (and we thank the referee for suggesting this)
whether this inequality remains valid for our extension of these
invariants. The following proposition gives an example which shows that the inequality must
be modified. It also motivates a conjecture as to what ought to hold. Recall that
according to Kanda \cite{Ka}, a Legendrian knot $\al$ in $(T^3,\xi_n)$
is quasilinear if $\al$ is isotopic on $T^3$ to a knot which lifts to a
straight line in the universal cover $\tilde T^3 = \Rset^3$.

\bpr For any $(p,q,r)\in\Zset^3\sm \{(0,0,0)\}$ there is
a quasilinear Legendrian knot $\al$ for the tight contact structure $(T^3,\xi_n)$ with a
$(p,q,r)$-periodic lift to $\Rset^3=\tilde T^3$ such that $\al$ satisfies the following equation:
\begin{equation}
tb(\al) + r(\al) = -\chi(\hat\Sigma) + rn. \label{Ben-eq}
\end{equation}
\epr

\begin{proof}
First, we note that for every $n>0$ and $(p,q,r)\in \Zset^3$
there exists a $(p,q,r)$-periodic Legendrian knot $\al$ in $(T^3,\xi_n)$, i.e., such that
$\tilde\al(t+1) = \tilde\al(t) + (p,q,r)$, where $\ta(t) =
(x(t),y(t),z(t))$ is the lift of $\al$ to $\Rset^3$. Furthermore, we may construct $\al$ so that for all
$t\in \Rset$ if $r>0$ (respectively, $r=0$ or $r<0$) we have $z'(t)>0$ (respectively, $z'(t)=0$ or $r'(t)<0$). We construct such a piecewise
linear Legendrian knot and then smooth it out by a small isotopy.
If $r>0$, take $t_0, t_1\in [0,1]$
with $t_0<t_1$ such that $\xi_n(t_0)$ is parallel to the $x$-axis and
$\xi_n(t_1)$ is parallel to the $y$-axis. Then define a piecewise
linear $(p,q,r)$-periodic Legendrian knot by letting $z(t)$ increase
linearly on the intervals $[0,t_0]$, $[t_0+\epsilon,t_1]$, and $[t_1+\epsilon, 1]$ modulo $1$ (for sufficiently small $\epsilon>0$, with $x(t)$ and $y(t)$ both constant
on these intervals,
while on the interval $[t_0,t_0+\epsilon]$ $x(t)$ increases by $p$ and on
$[t_1,t_1+\epsilon]$ $y(t)$ increases by $q$, with $z(t)$ constant.
Next, by a small Legendrian isotopy, deform this PL knot to a
smooth Legendrian knot $\al$ so as to preserve the property that $z'(t)>0$
for every $t$. The case $r<0$ is similar. For the case
$r=0$ we may take $\al$ to be the linear Legendrian knot $\al(t)=(pt,qt,z_0), t\in [0,1],$ where $z_0$ is such that the vector $Y=(p,q,0)\in \xi_n(x,y,z_0)$ for every $x,y\in\Rset$.

For such a Legendrian knot $\al$ with $r>0$ let us calculate the invariants. Since
$z'(t)>0$, the rotation number of the tangent vector $\al'$ with respect
to the constant Legendrian vector field $Z=\partial/\partial z$ in
$\xi_n$ is $r(\al)=0$. Next, take a vector $X\in \Rset^3$ orthogonal
to $(p,q,r)$ and construct a Seifert surface $\hat\Sigma\subset
\hat T^3$ such that along $\hat\al$ $X$ is tangent to $T\hat\Sigma$
and points inwards towards $\hat\Sigma$. Then it follows that $tb(\al)=rn$, since as $z(t)$ increases by $r$ the contact structure $\xi_n$ rotates
exactly $rn$ times. The Seifert surface $\hat\Sigma$ can
be taken to be homeomorphic to $S^1\times [0,\infty)$, so
$\chi(\hat\Sigma) = 0$. Thus equation (\ref{Ben-eq}) holds in
this case. For the case that $z<0$, consider $-\al$, the knot $\al$ with
the reversed orientation, and apply the case $r>0$. The
signs of $r(\al)$ and $r$ are reversed, while $tb(\al)$ and
$\chi(\hat\Sigma)$ continue to vanish, so the same formula holds.
The case $r=0$ is analogous, with $r(\al)=tb(\al)=0$.
\end{proof}
These examples motivate the following conjecture.
\bcj
For any Legendrian knot $\al$ on $(T^3,\xi_n)$ that is
$(p,q,r)$-periodic (in the positive direction of $\al$),
there is a Bennequin inequality
\begin{equation}
tb(\al) + r(\al) \leq -\chi(\hat\Sigma) + rn. \label{Ben-ineq}
\end{equation}
\ecj

We note that if $\al$ is a null-homologous Legendrian knot in
$(T^3,\xi_n)$, then $(p,q,r)=(0,0,0)$ and (\ref{Ben-ineq}) is
equivalent to the classical Bennequin inequality (\ref{Ben_classical}).
Furthermore, since every tight contact structure on $T^3$ is contactomorphic to some $\xi_n$,
the conjecture implies that a similar ineqality should hold for Legendrian knots in any tight contact structure $\xi$ on $T^3$,
provided $Z$ is taken to correspond to $\partial/\partial z$ under
the contactomorphism.







\begin{thebibliography}{00}


\bibitem[ Bennequin(1983)]{Be} D. Bennequin, {\it Entrelacement et \'equations de Pfaff}, Ast\'erisque {\bf 107--108} (1983), 83--161.


\bibitem[Burde and Zieschang(1985)]{BZ} G. Burde and H. Zieschang, {\it Knots}, De Gruyter Studies in Mathematics, 5. Walter de Gruyter, New York (1985).

\bibitem[Eliashberg(1984)]{E} Y. Eliashberg, {\it Classification of overtwisted contact
structures on $3$-manifolds}, Invent. Math. {\bf 98} (1989),
623-637.

\bibitem[Etnyre(2003)]{Et1} J. Etnyre, {\it Introductory lectures on contact geometry},
Topology and geometry of manifolds (Athens, GA,
2001), Proc. Sympos. Pure Math., {\bf 71}, Amer. Math. Soc., Providence, RI (2003), 81–-107.

\bibitem[Etnyre(2005)]{Et} J. Etnyre, {\it Legendrian and transversal knots}, in
Handbook of Knot Theory (W. Menasco e M. Thistlethwaite, eds.), Elsevier B. V., Amsterdam, 2005), 105-185,
arxiv 0306256v2.

\bibitem[Geiges(2004)]{Ge} H. Geiges, {\it Contact geometry}, in Handbook of Differential Geometry, vol. 2 (F.J.E. Dillen and L.C.A. Verstraelen, eds.), North-Holland, Amsterdam (2006), 315-382,
arxiv 0307242v2.

\bibitem[Ghiggini(2004)]{Gh} P. Ghiggini, {\it Linear Legendrian curves in $T^3$}, Math. Proc. Cambridge Phi. Soc. {\bf 140} (2006),
451-473, arxiv0402005v2.

\bibitem[Giroux(1994)]{Gi} E. Giroux {\it Une structure de contact, m\^eme tendue, est plus ou moins tordue}, Ann. Sci. \'Ecole Norm. Sup. (4) {\bf 27} (1994), 697-705, MR1307678.

\bibitem[Kanda(1997)]{Ka} Y. Kanda, {\it The classification of tight contact structures on the
3-torus}, Comm. Anal. Geom. {\bf 5} (1997), 413--438.

\bibitem[Martinet(1971)]{Ma} J. Martinet, {\it Formes de contact sur les vari\'et\'es de dimension 3}, in: Proc.
Liverpool Singularities Sympos. II, Lecture Notes in Math. {\bf 209}, Springer,
Berlin (1971), 142--163.

\bibitem[Souza(2007)]{So} F. S. Souza, {\it Legendrian Knots in $T^3$}, Masters Thesis, Pontif\'\i cia
Universidade Cat\'olica, Rio de Janeiro (2007), in Portuguese.



\end{thebibliography}



\end{document}